\documentclass[11pt]{article}

\usepackage{amsmath, amsfonts, amssymb}
\usepackage{mathrsfs}
\usepackage{theorem}

\usepackage{bm}
\newtheorem{thm}{Theorem}[section]
\newtheorem{prop}[thm]{Proposition}
\newtheorem{lem}[thm]{Lemma}

\newtheorem{defn}[thm]{Definition}

{\theorembodyfont{\rmfamily} \newtheorem{rem}[thm]{Remark}}
{\theorembodyfont{\rmfamily} \newtheorem{exam}{Example}[section]}

\newcommand{\ra}{\rightarrow}
\newcommand{\dis}{\displaystyle}

\def\R{\mathbb R}

\def\d{\text{\rm{d}}}
\def\E{\mathbb E}
\def\p{\mathbb P}

\def\la{\langle}
\def\raa{\rangle}
\def\La{\Lambda}
\def\veps{\varepsilon}

\def\diag{\mathrm{diag}}
\def\S{\mathcal S}

\newcommand{\fin}{\hspace*{\fill}\rule{0.3em}{1ex}}
\newenvironment{proof}{{\bf \noindent Proof.}}{\fin}

\numberwithin{equation}{section}

\begin{document}

\title{Ergodicity  of  regime-switching diffusions in Wasserstein distances \footnote{Supported in
 part by NSFC (No.11301030), 985-project and Beijing Higher Education Young Elite Teacher Project.}}

\author{Jinghai Shao\footnote{Email:\ shaojh@bnu.edu.cn}\\[0.6cm] {School of Mathematical Sciences, Beijing Normal University, 100875, Beijing, China}}
\date{March 2, 2014}
\maketitle
\begin{abstract}
Based on the theory of M-matrix and Perron-Frobenius theorem, we provide some criteria to justify the convergence of the regime-switching diffusion processes in  Wasserstein distances. The cost function we used to define the Wasserstein distance is not necessarily bounded. The continuous time Markov chains with finite and countable state space are all studied. To deal with the countable state space, we put forward a finite partition method. The boundedness for state-dependent regime-switching diffusions in an infinite state space is also studied.
\end{abstract}

\section{Introduction}
The regime-switching diffusion processes can be viewed as diffusion processes in random environments, which are characterized by continuous time Markov chains.
The behavior of the diffusion in each fixed environment may be very different.  Hence, they can provide more realistic models for many applications, for instance, control problems, air traffic management, biology and mathematical finance. We refer the reader to \cite{CDMR, Gh, J90, GZ, YZ} and references therein for more details of regime-switching diffusion processes and their applications.

In view of the  usefulness of regime-switching diffusion processes, the recurrent properties of these processes are rather complicated due to the appearance of the diffusion process and jump process at the same time. One can get this viewpoint from the examples constructed in \cite{PS}. In \cite{PS}, the authors showed that even in every fixed environment the corresponding diffusion process is recurrent (transient), the diffusion process in random environment could be transient (positive recurrent). In \cite{PP, PS, SX, YZ},  there are some studies on the transience, null recurrence, ergodicity, strong ergodicity of regime-switching diffusion processes. In these works, the convergence of the semigroups to their stationary distributions is in the total variation distance. Recently, in \cite{CH}, besides the total variation distance, the authors also considered the exponential ergodicity in the Wasserstein distance. They provided an on-off type criterion. In \cite{CH}, the state-independent and state-dependent regime-switching diffusion processes in a finite state space were studied. The cost function used in \cite{CH} to define the Wasserstein distance is bounded. The work \cite{CH} attracts us to studying the ergodicity of regime-switching diffusion processes in   Wasserstein distance.

In this work, we consider the regime-switching diffusion process $(X_t,\La_t)$ in the following form: $(X_t)_{t\geq  0}$ satisfies a stochastic differential equation (SDE)
\begin{equation}\label{1.1}
\d X_t=b(X_t, \La_t)\d t +\sigma(X_t,\La_t)\d B_t, \quad X_0=x\in \R^d,
\end{equation}
where $(B_t)$ is a $d$-dimensional Brownian motion, and  $(\La_t)$ is a continuous time  Markov chain with a state space $\mathcal S=\{1,2,\ldots, N\}$, $1\leq N \leq \infty$, such that
\begin{equation}\label{1.2}
\p(\La_{t+\delta}=l|\La_t=k, X_t=x)=\begin{cases} q_{kl}(x)\delta+o(\delta), &\text{if}\ k\neq l,\\
                                       1+q_{kk}(x)\delta+o(\delta), & \text{if}\ k=l,
                         \end{cases}
\end{equation}
for $\delta>0$. The $Q$-matrix $Q_x=(q_{kl}(x))$ is irreducible and conservative for each $x\in \R^d$.
If the $Q$-matrix $(q_{kl}(x))$ does not depend on $x$, then $(X_t,\La_t)$ is called a state-independent regime-switching diffusion; otherwise, it is called a state-dependent one. The state-independent switching $(\La_t)$ is also called Markov switching. When $N$ is finite, namely, $(\La_t)$ is a Markov chain in a finite state space, we call $(X_t,\La_t)$ a regime-switching diffusion process in finite state space. And it is easy to see the recurrent property of $(X_t,\La_t)$ is equivalent to that of $(X_t)$ as pointed out in \cite{PP}. When $N$ is infinite, we call $(X_t,\La_t)$ a regime-switching diffusion process in infinite state space.  There is few study on the recurrent property of the regime-switching processes in an infinite state space. The infinity of $N$ causes some well studied methods useless. In this setting, the effect of recurrent property of $(\La_t)$ to that of $(X_t,\La_t)$ is still not clear.

According to \cite[Theorem 2.1]{YZ}, we suppose that the following conditions hold throughout this work, which ensure that there exists a unique, non-explosive solution of (\ref{1.1}) and (\ref{1.2}): there exists $\bar K>0$ such that
\begin{itemize}
\item[$1^\circ$]\ \text{$q_{ij}(x)$ is a bounded continuous function for each pair of $i,j\in \S$;}
\item[$2^\circ$]\ $|b(x,i)|+\|\sigma(x,i)\| \leq \bar K(1+|x|),\quad x\in \R^d,\  i\in \S$;
\item[$3^\circ$]\ $|b(x,i)-b(y,i)|+\|\sigma(x,i)-\sigma(y,i)\|\leq \bar K|x-y|,\quad x,\,y\in \R^d,\ i\in \S$,
\end{itemize}
where $\|\sigma\|$ denotes the operator norm of matrix $\sigma$.

In this work, we establish some new criteria for the ergodicity of regime-switching processes in  Wasserstein distance.
We assume the semigroup $P_t$ of $(X_t,\La_t)$ converges weakly to some probability measure $\nu$, and consider under what condition $P_t$ also converges in Wasserstein distance to $\nu$. The existence of $\nu$ can be obtained by the results on the boundedness in moments of $(X_t,\La_t)$ (cf. \cite[Theorem 4.14]{Chen}). There are some studies on the asymptotic boundedness in moments for regime-switching processes in a finite state space. See, for instance, \cite{MY} for state-independent switching and \cite{YZ} for state-dependent switching. In Section 4, we shall discuss the boundedness for state-dependent regime-switching in an infinite state space.

In Section 3,  we first consider the ergodicity in  Wasserstein distance for state-independent regime-switching diffusion process in a finite state space. We provides a general result in Theorem \ref{main-1}. Based on it, we find three kinds of conditions to verify its assumption. In Theorem \ref{t-f-m}, we provide an easily verifiable condition using the theory of M-matrix. In Theorem \ref{prin}, we relate it to the well studied topic on the estimate of lower bound of principal eigenvalue of Dirichlet form when $(\La_t)$ is reversible. In Theorem \ref{t-f-m-b}, we give a concise condition using the Perron-Frobenius theorem, which recover the on-off type criterion established in \cite{CH}. But different to \cite{CH}, the cost function used by us to define the Wasserstein distance is not necessarily bounded.

In Section 3,  we proceed to study the state-independent regime-switching diffusion process in an infinite state space. We put forward a finite partition procedure to study the ergodicity in Wasserstein distance. By this method, we first divide the infinite state space $\S$ into finite number of subsets. Then we construct a new Markov chain in a finite state space, which induces a new regime-switching diffusion process. We show that if the new regime-switching process is ergodic in Wasserstein distance according to our criterion then so is the original one. Moreover, this finite projection method owns some kind of consistency, which is explained in Proposition \ref{refine} below. Examples are constructed to show the usefulness of our criteria.

In Section 4, we study the ergodicity of state-dependent regime-switching process in an infinite state space. For the state-dependent switching, it is more difficult to construct coupling process to estimate the Wasserstein distance.  See \cite{XS} for some study on this topic.  At the present stage, we provide a criterion by M-matrix theory to guarantee the existence  of the stationary distribution in weak topology.  We extend the result in \cite{YM03} on boundedness of state-independent switching in a finite state space to a state-dependent switching in an infinite state space. Moreover, this section also serves as providing method to check our assumption (A2) (A3) used in Section 3. At last, note that the convergence  in the weak topology is equivalent to the convergence in Wasserstein distance if the Wasserstein distance is induced by a bounded cost function.

This work is organized as follows. In Section 2, we give out some introduction on the Wasserstein distances, M-matrix theory and optimal couplings for diffusion processes. In Section 3, we provide a criterion on the exponential ergodicity of state-independent regime-switching process in a finite state space. The Section 4 is devoted to studying the state-independent regime-switching process in an infinite state space. In Section 5, we study the existence of stationary distribution of  state-dependent regime-switching processes in a finite state space.

\section{Preliminaries}
Let $(X_t,\La_t)$ be defined by (\ref{1.1}) and (\ref{1.2}). Letting $x=(x_1,\ldots,x_d),\,y=(y_1,\ldots,y_d)\in \R^d$, their Euclidean distance is $|x-y|:=\sqrt{\sum_{i=1}^d(x_i-y_i)^2}$. Let $\rho:[0,\infty)\ra [0,\infty)$ satisfying $\rho(0)=0$, $\rho'>0$ , $\rho''\leq 0$ and $\lim_{r\ra \infty} \rho(r)=\infty$. Then $(x,y)\mapsto\rho(|x-y|)$ is a new metric on $\R^d$. Replacing the original Euclidean distance to the new distance $\rho(|x-y|)$ is useful in application. See, for instance, \cite{Ch-94} for application in estimating the spectral gap of Laplacian operator on manifolds.  Define two distances $\tilde \rho$ and $\tilde \rho_b$ on $\R^d\times \mathcal S$ by
\begin{equation}\label{dist-1}
\tilde\rho((x,i),(y,j))=\sqrt{\mathbf 1_{i\neq j}+\rho(|x-y|)},\quad x,\,y\in \R^d,\, i,\,j\in \mathcal S,
\end{equation}
and
\begin{equation}\label{dist-2}
\tilde \rho_p((x,i),(y,j))=\sqrt{\mathbf 1_{i\neq j}+\rho^p(|x-y|) },\quad x,\,y\in \R^d,\, i,\,j\in \mathcal S, \ p>0.
\end{equation}
Let $\mathcal P(\R^d\times \mathcal S)$ be the collection of all probability measures on $\R^d\times \mathcal S$. Using $\tilde \rho $ and $\tilde\rho_p$ to be the cost function separatively, we can define the Wasserstein distance between every two probability measures $\mu$ and $\nu$ in $\mathcal P(\R^d\times \mathcal S)$ by
\begin{equation}\label{W-dist}
W_{\tilde \rho}(\mu,\nu)=\inf\big\{\E[\tilde\rho(X_1,X_2)]\big\},\quad W_{\tilde \rho_p}(\mu,\nu)=\inf\big\{\E[\tilde \rho_p(X_1,X_2)]\big\},\ \ p>0,
\end{equation}
where the infimum is taken over all pairs of random variables $X_1$, $X_2$ on $\R^d\times \mathcal S$ with respective laws $\mu$, $\nu$.

As our criteria on the ergodicity of $(X_t,\La_t)$ are related to the theory of M-matrix,  here we introduce some basic definition and notation of M-matrices, and refer the reader to the book \cite{BP} for more discussions on this well studied topic. The theory of M-matrix has also been used to study the stability of state-independent regime-switching processes in a finite state space (see \cite[Theorem 5.3]{MY}).

Let $B$ be a matrix or vector. By $B\geq 0$ we mean that all elements of $B$ are non-negative. By $B>0$ we mean that $B\geq 0$ and at least one element of $B$ is positive. By $B\gg 0$, we mean that all elements of $B$ are positive. $B\ll 0$ means that $-B\gg 0$.
\begin{defn}[M-matrix] A square matrix $A=(a_{ij})_{n\times n}$ is called an M-Matrix if $A$ can be expressed in the form $A=sI-B$ with some $B\geq 0$ and $s\geq\mathrm{Ria}(B)$, where $I $ is the $n\times n$ identity matrix, and $\mathrm{Ria}(B)$ the spectral radius of $B$. When $s>\mathrm{Ria}(B)$,   $A$ is called a nonsingular M-matrix.
\end{defn}
In \cite{BP}, the authors gave out 50 conditions which are equivalent to $A$ is a nonsingular M-matrix. We cite some of them below.
\begin{prop}\label{m-matrix}
The following statements are equivalent.
\begin{enumerate}
\item $A$ is a nonsingular $n\times n$ M-matrix.
\item All of the principal minors of $A$ are positive; that is,
\[\begin{vmatrix} a_{11}&\ldots&a_{1k}\\ \vdots& &\vdots\\ a_{1k}&\ldots&a_{kk}\end{vmatrix}>0 \ \  \text{for every $k=1,2,\ldots,n$}.\]
\item Every real eigenvalue of $A$ is positive.
\item $A$ is semipositive; that is, there exists $x\gg 0$ in $\R^n$ such that $Ax\gg0$.
\item There exists $x\gg 0$ with $Ax>0$ and $\sum_{j=1}^ia_{ij}x_j>0$, $i=1,\ldots, n$.
\end{enumerate}
\end{prop}

Now we recall some results on the optimal couplings of Wasserstein distances, which will help us to check the assumption (A1) below.
Let $(E,\tilde d,\mathcal E)$ be a complete separable metric space. Let $\mathcal P(E)$ denote the set of all probability measures on $E$. For two given probability measures $\mu$ and $\nu$ on $E$, define
\[W_p(\mu,\nu)=\inf_{\pi\in \mathcal(\mu,\nu)}\Big\{\int_{E\times E}\tilde d(x,y)^p\pi(\d x,\d y)\Big\}^{1/p},\quad p\geq 1,
\] where $\mathcal C(\mu,\nu)$ denotes the set of all couplings of $\mu$ and $\nu$.
It is well known that given $\mu$ and $\nu$, the infimum is attained for some coupling $\pi$. For a sequence of probability measures $\mu_n$ in $\mathcal P (E)$, the statement that $\mu_n$ converges to some $\mu\in \mathcal P(E)$ in the metric $W_p$ is equivalent to that $\mu_n$ converges weakly to $\mu$ and for some (or any) $a\in E$,
\[\lim_{K\ra \infty} \sup_n\int_{\{x:\tilde d(x,a)>K\}}\tilde d (x,a)^p\mu_n(\d x)=0.\] Hence, when $\tilde d$ is bounded, the convergence  in Wasserstein distance $W_p$ is equivalent to the weak convergence.

It can be seen from the definition of the Wasserstein distance that the calculation of this distance between two probability measures is not an easy task. Usually, one tries to find some suitable estimates on it.
In \cite{Ch-94}, some constructions of optimal couplings for Markov chain and diffusion processes are given. We recall some results on the optimal couplings of diffusion processes.

Consider a diffusion process in $\R^d$ with operator $L=\frac 12\sum_{i,j=1}^d a_{ij}(x)\frac{\partial^2}{\partial x_i\partial x_j}+\sum_{i=1}^db_i(x)\frac{\partial }{\partial x_i}$. For simplicity, we write $L\sim (a(x),b(x))$. Given two diffusions with operators $L_k\sim(a_k(x),b_k(x))$, $k=1,2$ respectively, an operator $\tilde L$ on $\R^d\times \R^d$ is called a coupling operator of $L_1$ and $L_2$ if
\begin{align*}
\tilde L f(x,y)&=L_1f(x), \ \text{if $f\in C_b^2(\R^d)$  and independent of $y$},\\
\tilde L f(x,y)&=L_2f(y), \ \text{if $f\in C_b^2(\R^d)$  and independent of $x$}.
\end{align*}
Let $\tilde d\in C^2(\R^d\times\R^d\backslash\{(x,x);x\in\R^d\})$ be a metric on $\R^d$.
A coupling operator $\bar L$ is called $\tilde d$-optimal if $\bar L\tilde d(x ,y)=\inf_{\tilde L}\tilde L\tilde d(x,y)$ for all $x\neq y$, where $\tilde L$ varies over all coupling operators of $L_1$ and $L_2$. The coefficients of any coupling operator must be of the form $\tilde L\sim \big(a(x,y), b(x,y)\big)$,
\[a(x,y)=\begin{pmatrix} a_1(x) & c(x,y)\\
                             c(x,y)^\ast &a_2(y)\end{pmatrix},\  b(x,y)=\begin{pmatrix} b_1(x)\\ b_2(y)\end{pmatrix},
                             \]
where $c(x,y)$ is a matrix such that $a(x,y)$ is non-negative definite and $c(x,y)^\ast$ denotes the transpose of $c(x,y)$. Next, we recall a result for optimal couplings of one-dimensional diffusion processes due to \cite[Theorem 5.3]{Ch-94}, and refer to \cite{Ch-94} for results in multidimensional case and more discussion.
\begin{thm}[\cite{Ch-94} Theorem 5.3]
Let $\rho\in C^2(\R_+,\R_+)$ with $\rho(0)=0$, $\rho'>0$ and $\rho''\leq 0$. Set $\tilde d(x,y)=\rho(|x-y|)$. If $d=1$, then the $\tilde d$-optimal solution $c(x,y)$ is given by $c(x,y)=-\sqrt{a_1(x)a_2(y)}$.
\end{thm}

\section{ Markovian switching in a finite state space}
Let $(X_t,\La_t)$ be a regime-switching diffusion process defined by (\ref{1.1}) and (\ref{1.2}) with $N<\infty$ and $Q=(q_{ij})$ state-independent switching.
For each fixed environment $i\in \mathcal S$, the corresponding diffusion $X_t^{(i)}$ is defined by
\begin{equation}\label{diff-fix}
\d X_t^{(i)}=b(X_t^{(i)},i)\d t+\sigma(X_t^{(i)},i)\d B_t, \quad X_0^{(i)}=x\in \R^d.
\end{equation}
Let $a^{(i)}(x)=\sigma(x,i)\sigma(x,i)^\ast$, then the infinitesimal operator $L^{(i)}$  of $(X_t^{(i)})$ is
\[L^{(i)}=\frac 12\sum_{k,l=1}^da^{(i)}_{kl}(x)\frac{\partial^2}{\partial x_k\partial x_l} +\sum_{k=1}^d b_k(x,i)\frac{\partial}{\partial x_k}.\]
Let $\rho:[0,\infty)\ra [0,\infty)$ satisfying
\begin{equation}\label{con-rho}
\rho(0)=0,\ \rho'>0,\ \rho''\leq 0\ \text{and}\  \rho(x)\ra \infty \ \text{as}\ x\ra \infty.
\end{equation}
It is clear that $\rho(x)=x$ satisfying previous conditions. In the sequel, the function $\rho$ we used satisfies (\ref{con-rho}).
We pose the following assumption to study the ergodic property of $(X_t,\La_t)$ in this section.
\begin{itemize}
\item[(A1)] For each $i\in \mathcal S$, there exist a coupling operator $\tilde L^{(i)}$ of $L^{(i)}$ and itself, and a constant $\beta_i$ such that
\begin{equation}\label{coup-ine}
\tilde L^{(i)}\rho(|x-y|)\leq \beta_i\rho(|x-y|), \quad x,\, y\in \R^d,\ x\neq y.\end{equation}
\item[(A2)] There exists  constants $C_1, \gamma>0$ such that
$\E[\rho(|X_t|)]\leq C_1\big(1+\E[\rho(|X_0|)]e^{-\gamma t}\big),\, t>0$.
\end{itemize}
Note that (A2) can not be deduced from (A1) by setting $y=0$ even when $\beta_i<0$.
According to \cite[Section 3]{ChL}, without loss of  generality, we assume that once the coupling processes corresponding to $\tilde L^{(i)}$ meet each other, then they will move together. This makes us to consider the inequality of (A1) only for $x\neq y$ in $\R^d$. This property will be used directly later, rather than mentioning the details.
Our introduction of $\rho(|\cdot|)$-optimal coupling operator could be used here to help us check this assumption. Note that the constant $\beta_i$ could be positive and negative similar to \cite{CH}.  If $\beta_i$ is negative, condition (A1) implies that the process $(X_t^{(i)})$ is exponentially convergent (see \cite{MT3}). When $\beta_i<0$ is smaller, the process $(X_t^{(i)})$ converges more rapidly. These constants $\beta_i$ are used to characterize the ergodic behavior of (\ref{coup-ine}). In \cite{PS},  R. Pinsky and M. Scheutzow constructed examples on $[0,\infty)\times \{1,2\}$ to show that even when
all $(X^{(k)}_t)$, $k=1,2$, are positive recurrent (transient), $(X_t)$ could be transient (positive recurrent).
These examples reveal the complexity in studying the recurrence properties of regime-switching
diffusions. One aim of this work is to show how the coaction of jumping process and diffusion process in each fixed environment determines the recurrent property of diffusion process in random environment. Besides, in \cite{SX}, we showed that in one dimensional space, if for each fixed environment $i\in \S$ with $N<\infty$, $(X_t^{(i)})$ is strongly ergodic, then so is $(X_t)$.

Recall the definition of Wasserstein distance $W_{\tilde \rho}$, $W_{\tilde \rho_p}$ given by (\ref{W-dist}). In this work, we write $\diag(\beta_1,\ldots,\beta_N)$ to denote the diagonal matrix induced by vector $(\beta_1,\ldots, \beta_N)^\ast$ as usual.
We now come to our first main result.

\begin{thm}\label{main-1}
Assume that (A1) (A2) and (A3) hold. If there exists a vector $\xi\gg0$ such that $\lambda=(\lambda_1,\ldots,\lambda_N)^\ast:=\big(Q+\diag(\beta_1,\ldots,\beta_N)\big)\xi\ll 0$, then there exists a probability measure $\nu$ on $\R^d\times \S$ such that
\begin{equation}\label{m-ine-1}
W_{\tilde \rho}(\delta_{(x,i)}P_t,\nu)\leq 2\tilde C(\sqrt{3+\rho(|x|)}+\tilde C) e^{-\tilde \alpha t},
\end{equation}
where $P_t$ is the Markovian semigroup associated with $(X_t,\La_t)$, $\delta_{(x,i)}$ denotes the Dirac measure at $(x,i)$, $\tilde \alpha $ and $\tilde C$ are positive constants, defined by $\tilde \alpha=\min\{\alpha,\theta\}/4$, $\tilde C=\max\{C_1, C_2, 1\}$ with $\alpha,\,\theta, C_2$ given by (\ref{theta}) and Lemma \ref{lem-1} below.
\end{thm}

The existence of $\xi\gg 0$ such that $\big(Q+\diag(\beta_1,\ldots,\beta_N)\big)\xi\ll 0$ is not an easily checked condition in general. Therefore, in Theorem \ref{t-f-m} below, we provide a sufficient condition by using the theory of M-matrix. When $(\La_t)$ is reversible, we relate the existence of such $\xi$ with the positiveness of principal eigenvalue in Theorem \ref{prin} below. In Theorem \ref{t-f-m-b}, we modify the definition of Wasserstein distance and provide another sufficient condition by using Perron-Frobenius theorem.

Before proving Theorem \ref{main-1}, we make some necessary preparation. First we construct the coupling process $(Y_t,\La'_t)$ of $(X_t,\La_t)$ to be used in the arguments of this Theorem.
Let $(\La_t, \La_t')$ be the classical coupling for $(\La_t)$, whose infinitesimal operator is  defined by
\begin{equation}\label{jump-coup}
\begin{split}
\tilde Qf(k,l)&:= \sum_{m,n\in \mathcal S} q_{(k,l)(m,n)}(f(m,n)-f(k,l))\\
       &=\mathbf 1_{\{k=l\}}\!\sum_{m\in \mathcal S} q_{km}(f(m,m)-f(k,l))+\mathbf 1_{\{k\neq l\}}\!\!\sum_{m\in \mathcal S,m\neq k}\!\!q_{km}(f(m,l)-f(k,l))\\
       &\quad+\mathbf 1_{\{k\neq l\}}\!\sum_{m\in \mathcal S,m\neq l}\!\!q_{lm}(f(k,m)-f(k,l)),
\end{split}
\end{equation}
for every measurable function $f$ on $\mathcal S\times \mathcal S$. This implies that once $\La_t=\La'_t$, then $\La_s=\La_s'$ for all $s>t$.
For $(x,k,y,l)\in \R^d\times \mathcal S\times \R^d\times \mathcal S$, set
\[a(x,k,y,l)=\mathbf 1_{\Delta}(k,l)a_k(x,y)+\mathbf 1_{\Delta^c}(k,l)\begin{pmatrix} a(x,k)&0\\0&a(y,l)\end{pmatrix},\]
where $\Delta=\{(k,k);\ k\in \S\}$, and $a_k(x,y)$ is determined by the coupling operator $\tilde L^{(k)}$ provided $\tilde L^{(k)}\sim (a_k(x,y), (b_k(x),b_k(y))^\ast)$.
Let $(X_t,Y_t)$ satisfy the following SDE,
\begin{equation}\label{coup-pro}
\d \begin{pmatrix} X_t\\ Y_t\end{pmatrix}=\Psi(X_t,\La_t,Y_t,\La'_t)\d W_t+\begin{pmatrix} b(X_t,\La_t)\\
b(Y_t,\La'_t)\end{pmatrix}\d t,
\end{equation}
where $\Psi(x,k,y,l)$ is a $2d\times 2d$ matrix such that $\Psi(x,k,y,l)\Psi(x,k,y,l)^\ast=a(x,k,y,l)$, and $(W_t)$ is a Brownian motion on $\R^{2d}$.

Let $\tau=\inf\{t\geq 0;\ \La_t=\La_t'\}$ be the coupling time of $(\La_t,\La'_t)$. Since $\mathcal S$ is a finite set, and  $\tilde Q$ defined by (\ref{jump-coup}) is irreducible, it is well known that there exists a constant $\theta>0$ such that
\begin{equation}\label{theta}\p(\tau>t)\leq e^{-\theta t},\quad t>0.\end{equation}
The processes $(X_t)$ and $(Y_t)$ defined by (\ref{coup-pro}) will evolve independently until time $\tau$, then they evolve as the diffusion process $(X_t^{(k)}, Y_t^{(k)})$ corresponding to $\tilde L^{(k)}$ when $\La_t=\La_t'=k$. So once $X_t=Y_t$ at some   $t>\tau$,  they will move together after $t$.  This kind of coupling processes has appeared in \cite{XS} and \cite{CH}. In \cite{XS}, together with F. Xi, we discussed the question   what conditions could ensure this kind of coupling to be successful when $(X_t,\La_t)$ is a state-dependent regime-switching process in a finite state space.

The following lemma is the key point in the argument of Theorem \ref{main-1}.
\begin{lem}\label{lem-1}
Let the assumption of Theorem \ref{main-1} be satisfied, and assume further  $\La_0=\La_0'$ and $X_0=x$, $Y_0=y$. Then for \,$C_2=\xi_{\mathrm{max}}/\xi_{\mathrm{min}}>0$, $\alpha=-\lambda_{\mathrm{max}}/\xi_{\mathrm{max}}>0$, where $\xi_{\mathrm{max}}=\max_{1\leq i\leq N}\xi_i$, $\xi_{\mathrm{min}}=\min_{1\leq i\leq N}\xi_i$, and
$\lambda_{\mathrm{max}}=\max_{1\leq i\leq N}\lambda_i<0$, it holds,
for every $t>s\geq 0$,
\begin{equation}\label{ine-1}
\E[\rho(|X_t-Y_t|)]\leq C_2\E[\rho(|X_s-Y_s|)]e^{-\alpha (t-s)}.
\end{equation}
\end{lem}

\begin{proof}
Since $\La_0=\La'_0$, by our construction of coupling  process $(\La_t,\La'_t)$, $\La_t=\La_t'$ for all $t>0$.
As $N<\infty$ and $\lambda\ll 0$, we have $ \lambda_{\mathrm{max}}=\max_{1\leq i\leq N} \lambda_i<0$.
Applying It\^o's formula (cf. \cite{Sko}),
we get
\begin{align*}
&\E[\rho(|X_t-Y_t|)\xi_{\La_t}]\\
&=\E[\rho(|X_s-Y_s|)\xi_{\La_s}]+\E\Big[\int_s^t(Q\xi)(\La_r)\rho(|X_r-Y_r|)+\tilde L^{(\La_r)} \rho(|X_r-Y_r|)\xi_{\La_r}\,\d r\Big]\\
&\leq \E[\rho(|X_s-Y_s|)\xi_{\La_s}]+\E\Big[\int_s^t\big((Q\xi)(\La_r)+\beta_{\La_r}\xi_{\La_r}\big)\rho(|X_r-Y_r|)\,\d r\Big]\\
&\leq \E[\rho(|X_s-Y_s|)\xi_{\La_s}]+ \lambda_{\mathrm{max}} \E\Big[\int_s^t \rho(|X_r-Y_r|)\d r\Big].
\end{align*}
This implies that
\begin{align*}
&\E[\rho(|X_t-Y_t|)\xi_{\La_t}]\leq \E[\rho(|X_s-Y_s|)\xi_{\La_s}]+\frac{ \lambda_{\mathrm{max}}}{\xi_{\mathrm{max}}}\E\Big[\int_s^t\rho(|X_r-Y_r|)\xi_{\La_r}\d r\Big].
\end{align*} The previous inequality still holds if we replace $s$ with $u$ satisfying $s<u<t$. So we can apply Gronwall's inequality in the differential form to obtain
\begin{equation*}
\E[\rho(|X_t-Y_t|)\xi_{\La_t}]\leq \E[\rho(|X_s-Y_s|)\xi_{\La_s}]e^{\lambda_{\mathrm{max}}(t-s)/\xi_{\mathrm{max}}},
\end{equation*}
and further
\begin{equation*}
\E[\rho(|X_t-Y_t|)]\leq \frac{\xi_{\mathrm{max}}}{\xi_{\mathrm{min}}}\E[\rho(|X_s-Y_s|)]e^{ \lambda_{\mathrm{max}}(t-s)/\xi_{\mathrm{max}}},
\end{equation*}
which yields the inequality (\ref{ine-1}).
\end{proof}

\noindent\textbf{Proof of Theorem \ref{main-1}.}\quad
Now we fix the initial point of the process $(X_t,\La_t,Y_t,\La_t')$ to be $(x,i,y,j)$ with $i\neq j$,
and go to estimate the Wasserstein distance between the distributions of $(X_t,\La_t)$ and $(Y_t,\La_t')$.
By the inequality (\ref{theta}) and Lemma \ref{lem-1}, we obtain
\begin{align*}\label{est-1}
&\E[\tilde \rho((X_t,\La_t), (Y_t,\La_t'))]\\ &=\E\big[\sqrt{\mathbf 1_{\La_t\neq \La'_t}+\rho(|X_t-Y_t|)}\mathbf 1_{\tau>t/2}\big]+\E\big[\sqrt{\rho(|X_t-Y_t|)}\mathbf 1_{\tau\leq t/2}\big]\\
&\leq \sqrt{\p(\tau>t/2)}\sqrt{\E[1+\rho(|X_t-Y_t|)]}+\sqrt{\E[\rho(|X_t-Y_t|)\mathbf 1_{\tau\leq t/2}]}\\
&\leq \sqrt{1+\E[\rho(|X_t|)+\rho(|Y_t|)]} e^{-\frac{\theta t}{4}}+\sqrt{\E[\E[\rho(|X_t-Y_t|)\big|\mathscr F_\tau]\mathbf 1_{\tau\leq t/2}]}\\
&\leq \sqrt{1+C_1(2+\rho(|x|)+\rho(|y|)) } e^{-\frac{\theta t}{4}}+\sqrt{C_2\E[\rho(|X_\tau-Y_\tau|)e^{\frac{-\alpha t}{2}}]}\\
&\leq \sqrt{1+C_1(2+\rho(|x|)+\rho(|y|)) }e^{-\frac{\theta t}{4}}+\sqrt{C_2C_1(2+\rho(|x|)+\rho(|y|)) }e^{\frac{-\alpha t}{4}}\\
&\leq 2\tilde C\sqrt{3+\rho(|x|)+\rho(|y|)}e^{ -\tilde \alpha t },
\end{align*} where $\tilde C=\max\{C_1, C_2,1\}$ independent of $(x,i,y,j)$ and $\tilde \alpha=\min\{\alpha,\theta\}/4>0$.
This implies
\begin{equation}\label{wass-ine-2}
W_{\tilde \rho}(\delta_{(x,i)}P_t,\delta_{(y,j)}P_t)\leq 2\tilde C\sqrt{3+\rho(|x|)+\rho(|y|)} e^{-\tilde \alpha  t}.
\end{equation}
%
By (A2), we know that $\E[\rho(|X_t|)]$ is bounded for all $t>0$. This yields that the family of probability measures $(\delta_{(x,i)}P_t)_{t>0}$ is weakly compact since for each $c>0$,  $\{x\in\R^d;\ \rho(|x|)\leq c\}$ is a compact set. Moreover,
\[\lim_{K\ra \infty}\sup_{t}\!\sum_{j\in\S}\int_{\tilde \rho((y,j),(x,i))\geq\! K} \!\!\!\tilde \rho((y,j),(x,i))\delta_{(x,i)}P_t(\d x,j)\leq \lim_{K\ra\infty}\sup_t\frac{\E[1+\rho(|x|)+\rho(|X_t|)]}{K}=0.\]
Hence, $(\delta_{(x,i)}P_t)_{t>0}$ is also compact in the Wassestein distance $W_{\tilde \rho}$.
There exists a subsequence $\delta_{(x,i)}P_{t_k}$, $t_k\ra \infty$ as $k\ra \infty$, converging in $W_{\tilde \rho}$-metric to some probability measure $\nu$ on $\R^d\times\S$. Inequality (\ref{wass-ine-2}) implies that for all $(y,j)\in \R^d\times \S$, $\delta_{(y,j)}P_{t_k}$ converges in $W_{\tilde \rho}$-metric to $\nu$, and further that $\nu_0P_{t_k}$ converges in $W_{\tilde \rho}$-metric to $\nu$ for every probability measure $\nu_0$ on $\R^d\times \S$ with $\displaystyle \int_{\R^d}\sqrt{\rho(|x|)}\nu_0(\d x,\S)<\infty$.
This yields that for each $s>0$, $\delta_{(x,i)}P_sP_{t_k}$ converges in $W_{\tilde \rho}$ to $\nu$. It is easy to see $\delta_{(x,i)}P_{t_k}P_s$ converges weakly to $\nu P_s$, hence converges in $W_{\tilde \rho}$-metric to $\nu P_s$. In all, we get for each $s>0$, $\nu P_s=\nu$, then $\nu$ is a stationary distribution of $P_t$.
As $(\delta_{(x,i)}P_{t_k})$ converges weakly to $\nu$, by (A2), we get by Fatou's lemma
\begin{equation}\label{supp-nu}
\int_{\R^d\times \S}\rho(|y|)\d \nu \leq \liminf_{k\ra\infty}\int_{\R^d\times \S}\rho(|y|)\d\big(\delta_{(x,i)}P_{t_k}\big)=\liminf_{k\ra \infty}\E[\rho(|X_t|)]\leq C_1.
\end{equation}
By (\ref{wass-ine-2}),
\begin{align*}
&W_{\tilde\rho}(\delta_{(x,i)}P_t,\nu)=W_{\tilde \rho}(\delta_{(x,i)}P_t,\nu P_t)\\
&=\sup_{\varphi:\mathrm{Lip}(\varphi)\leq 1}\Big\{\int_{\R^d\times \S} \varphi(y,j)\d \big(\delta_{(x,i)}P_t\big)-\int_{\R^d\times \S}\varphi(y,j)\d \big(\nu P_t\big)\Big\}\\
&\leq \int_{\R^d\times \S}\nu(\d y, j) W_{\tilde \rho}(\delta_{(x,i)}P_t,\delta_{(y,j)}P_t)\\
&\leq 2\tilde C(\sqrt{3+\rho(|x|)}+ \sqrt{C_1}) e^{-\tilde \alpha t}.
\end{align*}
Here we have used the duality formula for the Wasserstein distance, and
 \[\mathrm{Lip}(\varphi):=\sup\Big\{\frac{\varphi(y,j)-\varphi(z,k)}{\tilde \rho((y,j),(k,l))};\ (y,j)\neq (z,k) \Big\}.\]
Till now, we have completed the proof of this theorem.

\begin{thm}\label{t-f-m}
Assume that (A1) and (A2) hold. If the matrix $-\big(Q+\diag(\beta_1,\ldots,\beta_N)\big)$ is a nonsingular M-matrix, then there exists a probability measure $\nu$ on $\R^d\times \S$ such that
\begin{equation}\label{exp-ine-1}
W_{\tilde\rho}(\delta_{(x,i)}P_t,\nu)\leq 2\tilde C(\sqrt{3+\rho(|x|)}+\tilde C)   e^{-\tilde \alpha t },
\end{equation}
The constants $\tilde \alpha$ and  $\tilde C$ are defined  in Theorem \ref{main-1}.
\end{thm}

\begin{proof}
By Proposition \ref{m-matrix},
since $-(\diag(\beta_1,\ldots,\beta_N)+Q)$ is a nonsingular M-matrix, there exists a vector $\xi=(\xi_1,\ldots, \xi_N)^\ast\gg 0$ such that
\[\lambda:=(\diag(\beta_1,\ldots,\beta_N)+Q)\xi\ll 0.
\]
According to Theorem \ref{main-1}, the desired results hold.
\end{proof}

Next, we assume that $(\La_t)$ is reversible with $\pi=(\pi_i)$ being its reversible probability measure. So it holds $\pi_iq_{ij}=\pi_j q_{ji}$, $i,\,j\in \S$.
Let $L^2(\pi)=\{f\in\mathscr B(\S);\ \sum_{i=1}^N \pi_if_i^2<\infty\}$, and denote by $\|\cdot\|$ and $\la\cdot, \cdot\raa$ respectively the norm and inner product in $L^2(\pi)$.
Let
\begin{equation}\label{Dirichlet}
D(f)=\frac 12\sum_{i,j=1}^N\pi_iq_{ij}(f_j-f_i)^2-\sum_{i=1}^N\pi_i\beta_if_i^2,\quad f\in L^2(\pi),
\end{equation}
where $(\beta_i)$ is given by condition (A1).
We borrow the notation $D(f)$ from the Dirichlet theory for continuous time Markov chain, but we should note that in our case $\beta_i$ could be positive, so $D(f)$ may be negative which is different to the standard Dirichlet theory.
Define the principal eigenvalue by
\begin{equation}\label{prin-eig}
\lambda_0=\inf\big\{D(f);\ f\in L^2(\pi),\ \|f\|=1\big\}.
\end{equation}
\begin{thm}\label{prin}Let (A1)  and (A2) be satisfied and  assume that $(\La_t)$ is reversible with respect to the probability measure $(\pi_i)$. Assume the principal eigenvalue $\lambda_0>0$. 
Then there are positive constants $\tilde C$, $\tilde \alpha$ and a probability measure $\nu$ on $\R^d\times \S$ so that
\begin{equation}\label{exp}
W_{\tilde\rho}(\delta_{(x,i)}P_t,\nu)\leq 2\tilde C(\sqrt{3+\rho(|x|)}+\tilde C)  e^{-\tilde \alpha t }.
\end{equation}
\end{thm}

\begin{proof}
As $N<\infty$ and $\lambda_0>0$, there exists a $g\in L^2(\pi)$ such that $g\not \equiv 0$, $D(g)=\lambda_0\|g\|^2$.
We shall show that  $g\gg 0$ and $Qg(i)+\beta_ig_i=-\lambda_0 g_i$, $i\in\S$, then this theorem follows immediately from Theorem \ref{main-1} by taking $\xi=g$.
We use the variational method in \cite{Ch-00}.
It is easy to check $D(f)\geq D(|f|)$, so it must hold $g\geq 0$. For a fixed $k\in \S$, let $\tilde g_i=g_i$ for $i\neq k$ and $\tilde g_k=g_k+\veps$.
It holds $Q\tilde g(i)=Qg(i)+\veps q_{ik}$ for $i\neq k$ and $Q\tilde g(k)=Qg(k)-\veps q_k$.
We have
\begin{align*}
D(\tilde g)&=\la \tilde g,-Q\tilde g\raa-\sum_{i=1}^N\pi_i\beta_i\tilde g_i^2\\
&=\la  g, -Qg\raa-\sum_{i=1}^N\pi_i \beta_i  g_i^2 +2\veps \pi_k(-Qg)(k) -2\veps\pi_k\beta_kg_k-\veps^2\pi_k(q_k-\beta_k),
\end{align*}where we have used $\pi_iq_{ik}=\pi_kq_{ki}$.
Because $D(\tilde g)\geq \lambda_0\|\tilde g\|^2$ and $D(g)=\lambda_0\|g\|^2$, we get
\begin{equation*}
-2\veps \pi_k\big(\lambda_0g_k+Qg(k)+\beta_kg_k\big)+\veps^2\pi_k(g_k-\beta_k)-2\lambda_0\veps^2\pi_k\geq 0.
\end{equation*}
This yields $Qg(k)+\beta_kg_k=-\lambda_0g_k$ since $\veps$ is arbitrary, and then $Qg(i)+\beta_ig_i=-\lambda_0 g_i$ for each $i\in\S$ since $k$ is arbitrary.

Since $g\not\equiv 0$ and $g\geq 0$, there exists $k$ such that $g_k>0$. If $q_{ik}>0$, then
\[0<q_{ik}g_k\leq \sum_{j\neq i}q_{ij}g_j=(q_i-\beta_i-\lambda_0)g_i,\]
so $g_i>0$ and $q_i-\beta_i-\lambda_0>0$. As $Q$ is irreducible, by an inductive procedure, we can prove that $g_i>0$ for every $i\in \S$.
\end{proof}

\begin{rem}
  According to the argument of Theorem \ref{prin} and the statement 3 of Proposition \ref{m-matrix}, we obtain that $\lambda_0>0$ is equivalent to the statement $-(Q+\diag(\beta))$ is a nonsingular M-matrix when $(\La_t)$ is reversible process in a finite state space, i.e. $N<\infty$. However, the criterion expressed by the principal eigenvalue $\lambda_0$ of a bilinear form can be extended directly to deal with the situation $N=\infty$ and the criterion expressed by nonsingular M-matrix can not. To apply the criterion expressed by the principal eigenvalue,  one has to justify the positiveness of $\lambda_0$ which is not easy when $N=\infty$. But this is not the main topic of present work, and we are satisfied with this connection at present stage and leave further study of $\lambda_0$ to another work.
\end{rem}

If we use the metric $\tilde \rho_p$ on $\R^d\times \S$, we can recover the condition given by \cite[Theorem 1.4]{CH} to justify the exponential ergodicity of $(X_t,\La_t)$.
The advantage of this criterion (see (\ref{exp-erg-1}) below) is that it has very concise expression, and the disadvantage  is that we can not fix explicitly the power $p$.
\begin{thm}\label{t-f-m-b}
Assume that (A1)  and (A2) hold. Let $\mu=(\mu_i)_{i\in \S}$ be the invariant probability measure of $(q_{ij})$. If
\begin{equation}\label{exp-erg-1}
\sum_{i=1}^N\mu_i\beta_i<0,
\end{equation}
then there exist positive constants $p$, $\alpha_p$, $\tilde C_1$,   such that
\begin{equation}\label{exp-erg-2}
W_{\tilde\rho_p}(\delta_{(x,i)}P_t,\nu)\leq 2\tilde C_1(\sqrt{3+\rho(|x|)}+\tilde C_1) e^{-\alpha_p t},
\end{equation}
where $\tilde C_1$ is independent of  $(x,i)$.
\end{thm}

\begin{proof}
Let $Q_p=Q+p\,\diag(\beta_1,\ldots,\beta_N)$, and
\[\eta_p=-\max_{\gamma\in \mathrm{spec}(Q_p)}\mathrm{Re}\,\gamma, \quad \text{where $\mathrm{spec}(Q_p)$ denotes the spectrum of $Q_p$}.\]
Let $Q_{(p,t)}=e^{tQ_p}$, then the spectral radius $\mathrm{Ria}(Q_{(p,t)})$ of $Q_{(p,t)}$ equals to $e^{-\eta_p t}$. Since all coefficients of $Q_{(p,t)}$ are positive, Perron-Frobenius theorem (see \cite[Chapter 2]{BP}) yields $-\eta_p$ is a simple eigenvalue of $Q_p$. Moreover, note that the eigenvector of $Q_{(p,t)}$ corresponding to $e^{-\eta_p t}$ is also an eigenvector of $Q_p$ corresponding to $-\eta_p$. Then Perron-Frobenius theorem ensures that there exists an eigenvector $\xi\gg 0$ of $Q_p$ corresponding to $-\eta_p$. Now applying Proposition 4.2 of \cite{BGM} (by replacing $A_p$ there with   $Q_p$), if $\sum_{i=1}^N\mu_i\beta_i<0$, then there exists some $p_0>0$ such that $\eta_p>0$ for any $0<p<p_0$. Fix a $p$ with $0<p<\min\{1,p_0\}$ and an eigenvector $\xi\gg 0$, then we obtain
\[Q_p\,\xi=(Q+p\,\diag(\beta_1,\ldots,\beta_N))\xi=-\eta_p\,\xi\ll 0.\]

For the coupling operator $\tilde L^{(i)}\sim(a^{(i)}(x,y),b^{(i)}(x,y))$, due to the nonnegative definiteness of $a^{(i)}(x,y)$ and $ 0<p<1$, it can be checked by direct calculus that
\[\tilde L^{(i)} \rho^p(|x-y|)\leq p \rho^{p-1}(|x-y|)\tilde L^{(i)}\rho(|x-y|),\quad x,\,y\in \R^d,\ x\neq y.\]
Combining with Assumption (A1), we get
\begin{equation}\label{oper-ine-p}
\tilde L^{(i)}\rho^p(|x-y|)\leq p\beta_i\rho^p(|x-y|),\quad x,\,y\in\R^d,\ x\neq y.
\end{equation}
Similar to the argument of Lemma \ref{lem-1}, if $\La_s=\La_s'$ for some $0\leq s<t$, by It\^o's formula, we obtain
\begin{align*}
&\E[\rho^p(|X_t-Y_t|)\xi_{\La_t}]\\
&\leq\E[\rho^p(|X_s-Y_s|)\xi_{\La_s}]+\E\Big[\int_s^t\big((Q+p\,\diag(\beta_1,\ldots,\beta_N))\xi\big)(\La_r)\rho^p(|X_r-Y_r|)\d r\Big]\\
&\leq \E[\rho^p(|X_s-Y_s|)\xi_{\La_s}]-\eta_p\E\Big[\int_s^t\rho^p(|X_r-Y_r|)\xi_{\La_r}\d r\Big].
\end{align*}
Due to the arbitrariness of $s$, $0\leq s<t$, we can apply  Gronwall's inequality in differential form to get
\begin{align*}
\E[\rho^p(|X_t-Y_t|)\xi_{\La_t}]\leq \E[\rho^p(|X_s-Y_s|)\xi_{\La_s}]e^{-\eta_p(t-s)}.
\end{align*}
Consequently,
\begin{equation}\label{ine-3}
\E[\rho^p(|X_t-Y_t|)]\leq C_3\E[\rho^p(|X_s-Y_s|)]e^{-\eta_p(t-s)},
\end{equation}
where $C_3=\max_{k,l\in \S} \big(\xi_k/ \xi_l\big)\geq 1$.

Now we go to estimate the Wasserstein distance between the distributions of $(X_t,\La_t)$ and $(Y_t,\La_t')$ with $(X_0,\La_0,Y_0,\La_0')=(x,i,y,j)$ and $i\neq j$.
\begin{align*}
&\E\big[\tilde\rho_p((X_t,\La_t),(Y_t,\La_t'))\big]\\
&=\E\big[\sqrt{\mathbf 1_{\La_t\neq \La_t'}+\rho^p(|X_t-Y_t|)}\mathbf 1_{\tau>t/2}\big]+\E\big[\sqrt{\rho^p(|X_t-Y_t|)}\mathbf 1_{\tau\leq t/2}\big]\\
&\leq \sqrt{\p(\tau>t/2)}\sqrt{1+\big(\E[\rho(|X_t-Y_t|)])^{p}}+\sqrt{\E[\rho^p(|X_t-Y_t|)\mathbf 1_{\tau\leq t/2}]}\\
&\leq 2\tilde C_1\sqrt{3+\rho(|x|)+\rho(|y|)} e^{- \alpha_p t },
\end{align*} where $\tilde C_1=\max\{C_1, C_3\}$, $ \alpha_p=\min\{\theta,\eta_p\}/4>0$.
This yields that
\begin{equation}\label{coup-ine-2}W_{\tilde \rho_p}(\delta_{(x,i)}P_t,\delta_{(y,i)}P_t)\leq 2\tilde C_1\sqrt{3+\rho(|x|)+\rho(|y|)} e^{- \alpha_p t }.
\end{equation}
As in the late part of the argument of Theorem \ref{main-1}, (\ref{coup-ine-2}) can yield the desired result. The proof is completed.
\end{proof}

\begin{rem}
From the argument of Theorem \ref{t-f-m-b}, we essentially use the Perron-Frobenius theorem to ensure the existence of a vector $\xi\gg 0$ such that
$\big(Q+p\,\diag(\beta_1,\ldots,\beta_N)\big)\xi\ll 0$. Moreover, in the proof of \cite[Proposition 4.2]{BGM} the fact $-\eta_p=\pi_pQ_p\mathbf 1=\pi_p\,\diag(\beta_1,\ldots,\beta_N)\mathbf 1$ has been used for the left eigenvector $\pi_p$ of $Q_p$ associated to $-\eta_p$ with $\pi_p\mathbf 1=1$. This prevents us from applying this method to regime-switching processes with a countable state space $\S$.
\end{rem}

\section{Markovian Switching in a countable state space}
In this section, we consider the regime-switching diffusion $(X_t,\La_t)$ given by (\ref{1.1}) and (\ref{1.2}) with state-independent switching in a countable set, i.e. $\S=\{1,2,\ldots,N\}$ and $N=\infty$. There few result on the ergodicity of regime-switching diffusion process when the state space $\S$ of $(\La_t)$ is a countable set. In this case, the criteria expressed by Lyapunov function or drift condition for general Markov processes still work. But it is well known that constructing Lyapunov functions is a difficult job even for diffusion processes.  To construct a Lyapunov function  for a regime-switching diffusion becomes more difficult due to the appearance of diffusion operator and jump operator in its infinitesimal generator at the same time. In this part, we put forward a method to transform the switching process $(\La_t)$ in a countable state space into a new one in a finite state space. By using the criterion established in previous section by M-matrix theory, we can guarantee that if the new regime-switching diffusion process is ergodic in Wasserstein distance, then so is the original one.

In this section, we assume that  (A1) holds, and   further  $M:=\sup_{i\in \S}\beta_i<\infty$. As $\S$ is a countable set, we need more assumption on $(\La_t)$ so that the coupling method could be applied.
\begin{itemize}
\item[(A3)] The $Q$-matrix of $(\La_t)$ is conservative irreducible and  $\sup_{i\in\S} q_i<\infty$.
There is a coupling process $(\La_t,\La_t')$ with operator $\tilde Q$ on $\S\times \S$. Suppose there is a bounded function $g\geq 0$ in the domain of $\tilde Q$ such that $g(i,i)=0$ and
\begin{equation}\label{c-o}\tilde Qg(i,j)\leq -1,\quad i\neq j.\end{equation}
\end{itemize}
According to \cite[Theorem 5.18]{Ch-2}, (A2) implies that for $0<\theta<\|g\|_{\infty}$,
\begin{equation}\label{coup-time}
\tilde E[ e^{\theta\, \tau}]\leq \frac{1}{1-\theta\, \|g\|_{\infty}},
\end{equation}
where $\tau=\inf\{t\geq 0,\La_t=\La_t'\}$. We choose and fix a $\theta$ with $0<\theta<\|g\|_{\infty}$ in the rest of this section. The inequality (\ref{coup-time}) is what we need to estimate the Wasserstein distance directly. Assumption (A2) provides a sufficient condition to guarantee (\ref{coup-time}) hold.
The coupling process for a continuous time Markovian chain is a well studied topic. There are lots of work on this topic. For example, due to \cite{Zhang, Mao2}, there are explicit conditions in terms of birth rate and death rate to check condition (A2) for birth-death process. We refer the reader to  \cite[Chapter 5]{Ch-2} for more discussion on this condition.

First, we divide $\S$ into finite subsets according to $\beta_i$. Precisely, choose a finite partition $\Gamma$ of $(-\infty,M]$, that is,
\[\Gamma:=\big\{ -\!\infty=:k_0<k_1<\cdots<k_m<k_{m+1}:=M\big\}.\] Corresponding to  $\Gamma$, there is a finite partition of  $\S$,
denoted by $ F:= \{F_1,\ldots,  F_{m+1}\}$, where \[F_i=\big\{j\in\S;\ \beta_j\in (k_{i-1},k_{i}]\big\},\quad i=1,\ldots, m+1.\]
We assume that each $F_i$ is not empty, otherwise, we can delete some points in the partition $\Gamma$ to ensure it.
Let $\phi:\S\ra \{1,\ldots,m+1\}$ be a  map  defined by $\phi(j)=i$ if $j\in F_i$.
Let
\begin{equation}\label{beta}
\beta^F_i=\sup_{j\in F_i}\beta_{j} \ \text{for}\  i=1,\ldots,m+1, \ \text{so}\ \beta_j\leq \beta_{\phi(j)}^F \ \text{for every  $j\in \S$}, \ \text{and}\ \beta_i^F<\beta_{i+1}^F.
\end{equation}
Set $Q^F=(q_{ij}^F)$ be a new $Q$-matrix on state space $\{1,2,\ldots,m+1\}$ corresponding to $F$ defined by
\begin{equation}\label{Q-f}
q_{ik}^F=\inf_{r\in F_i}\sum_{j\in F_k} q_{rj},\  k>i;\quad q_{ik}^F=\sup_{r\in F_i}\sum_{j\in F_k} q_{rj},\  k<i,\ \text{and}\ q_{ii}^F=-\sum_{k\neq i} q_{ik}^F.
\end{equation}
As each $F_i$ is nonempty and $(q_i)_{i\in \S}$ is bounded, we get $0\leq q_{ik}^F\leq \sup_{i\in \S} q_i<\infty,\, k\neq i$.  It is not easy to check whether $Q^F$ is irreducible. But this does not impact the criterion provided  by the theory of M-matrix. It is an advantage that there is no demand on irreducibility  in checking a matrix to be nonsingular  M-matrix.  However, in the study of Perron-Frobenius theorem,  irreducibility of a matrix plays important role.

\begin{thm}\label{t-m-infi}
Assume that (A1) (A2)   and (A3) hold. For the partition $F$ given above, if the $(m+1)\times (m+1)$-matrix
\[-\big(Q^F+\diag(\beta_1^F,\ldots,\beta_{m+1}^F)\big)H_{m+1}\]
is a nonsingular M-matrix, where
\begin{equation}\label{h-matrix}H_{m+1}=\begin{pmatrix}1&1&1&\cdots&1\\ 0&1&1&\cdots&1\\ \vdots&\vdots&\vdots&\cdots&\vdots\\ 0&0&0&\cdots&1\end{pmatrix}_{(m+1)\times (m+1)},
\end{equation}
then there exist  constants $\tilde C,\,\tilde \alpha>0$ and a probability measure $\nu$ on $\R^d\times \S$ such that
\[W_{\tilde \rho}(\delta_{(x,i)}P_t,\nu)\leq  2\tilde C(\sqrt{3+\rho(|x|)}+\tilde C)  e^{-\tilde \alpha t},\quad (x,i)\in \R^d\times \S,
\] where $\tilde C>0$ is independent of  $(x,i)$.
\end{thm}

\begin{proof}
Let $(\La_t,\La'_t)$ be the coupling given by (A2). Let $(X_t,Y_t)$ be defined by (\ref{coup-pro}). Similar to the proof of Theorem \ref{t-f-m}, the key point is also the estimates given by Lemma \ref{lem-1}.

As $-\big(Q^F+\diag(\beta_1^F,\ldots,\beta_{m+1}^F)\big)H_{m+1}$ is a nonsingular M-matrix, there exists a vector $\eta^F=(\eta_1^F,\ldots,\eta_{m+1}^F)^\ast\gg 0$ such that
\[\lambda^F=(\lambda_1^F,\ldots,\lambda_{m+1}^F)^\ast=\big(Q^F+\diag(\beta_1^F,\ldots,\beta_{m+1}^F)\big)H_{m+1}\eta^F\ll 0.\]
Then $\bar \lambda:=\max_{1\leq i\leq m+1} \lambda_i^F<0$.
Set $\xi^F=H_{m+1}\eta^F$.
Then
\[\xi_i^F=\eta_{m+1}^F+\cdots+\eta_i^F,\ i=1,\ldots m+1. \]
Hence, $\xi_{i+1}^F<\xi_i^F$, $i=1,\ldots,m$, and $\xi^F\gg 0$.

We extend the vector $\xi^F$ to a vector on $\S$ by setting $\xi_r=\xi_i^F$, if $r\in F_i$.
For $r\in F_i$,  we obtain
\begin{align*}
Q\xi(r)&=\sum_{j\in\S,j\neq r} q_{rj}(\xi_j-\xi_r)=\sum_{j\notin F_i,j\in \S}q_{rj}(\xi_j-\xi_r)\\
&=\sum_{k<i}\big(\sum_{j\in F_k} q_{rj}\big)(\xi_k^F-\xi_i^F)+\sum_{k>i}\big(\sum_{j\in F_k}q_{rj}\big)(\xi_k^F-\xi_i^F)\\
&\leq \sum_{k<i} q_{ik}^F(\xi_k^F-\xi_i^F)+\sum_{k>i}q_{ik}^F(\xi_k^F-\xi_i^F)
=\big(Q^F\xi^F\big)(i),
\end{align*}
where we have used (\ref{Q-f}).
Applying It\^o's formula to $(X_t,\La_t)$, $(Y_t,\La_t')$ with  $X_0=x$, $Y_0=y$ and  $\La_0=\La_0'=r$, we have, for every $0<u<t$,
\begin{align*}
&\E[\rho(|X_t-Y_t|)\xi_{\La_t}]\\
&\leq \E[\rho(|X_u-Y_u|)\xi_{\La_u}]+\E\Big[\int_u^t\big((Q \xi)(\La_s)+\beta_{\La_s}\xi_{\La_s}\big)\rho(|X_s-Y_s|)\d s\Big]\\
&\leq \E[\rho(|X_u-Y_u|)\xi_{\La_u}]+\E\Big[\int_u^t\big((Q^F\xi^F)(\phi(\La_s))+\beta^F_{\phi(\La_s)}\xi_{\phi(\La_s)}^F\big)\rho(|X_s-Y_s|)\d s\Big]\\
&\leq \E[\rho(|X_u-Y_u|)\xi_{\La_u}]+\frac{\bar \lambda}{\xi_{\mathrm{min}}^F}\E\Big[\int_u^t\rho(|X_s-Y_s|)\xi_{\La_s}\d s\Big],
\end{align*}
where $\phi:\S\ra F$ denotes the projection map, $\xi_{\mathrm{max}}^F=\max_{1\leq i\leq m+1}\xi^F_i>0$. Set $\tilde \alpha=-\bar\lambda/\xi_{\mathrm{max}}^F>0$.
Due to the arbitrariness of $0<u<t$, we can apply Gronwall's inequality in differential form to get
\begin{equation}\label{ine-5}
\E\big[\rho(|X_t-Y_t|)\xi_{\La_t}\big]\leq \E\big[\rho(|X_u-Y_u|)\xi_{\La_u}\big]e^{-\tilde \alpha(t-u)},\quad 0<u<t.
\end{equation}
Thanks to (\ref{ine-5}) and (A2), we can prove that
\[W_{\tilde \rho}(\delta_{(x,i)}P_t,\delta_{(y,j)}P_t)\leq 2\tilde C\sqrt{3+\rho(|x|)+\rho(|y|)} e^{-\tilde \alpha t}\]
for some $\tilde C>0$ and $\tilde \alpha>0$.
According to (A3), we know that $(\La_t)$ is exponential ergodic, hence there exists a compact function $h$ on $\S$ such that $\sup_{t>0}\E[h(\La_t)]\leq C_4$, where $C_4$ is a positive constant (see \cite[Theorem 4.4]{Chen}). Combining with (A2), there is a constant $C_5>0$ so that
\[\sup_{t>0}\E[\rho(|X_t|)+h(\La_t)]\leq C_5.\]
Therefore, the family of probability measures $(\delta_{(x,i)}P_t)_{t>0}$ is weakly compact.
Then, following the similar argument as in Theorem \ref{main-1}, we can conclude the proof.
\end{proof}

Now we consider the consistency of our method on the finite partitions.
Under the assumption (A1), consider two finite partitions $\tilde \Gamma$ and $\Gamma$ of $(-\infty,M]$ such that $\tilde \Gamma$ is a refinement of $\Gamma$. Associated with $\tilde \Gamma$ and $\Gamma$, there are respectively two finite partitions $\tilde F$ and $F$ of $\S$ given by
\[\tilde F=\{\tilde F_1,\ldots,\tilde F_{n+1}\},\ \text{and}\ F=\{F_1,\ldots,F_{m+1}\}.\]
Therefore, each $\tilde F_k$ is a subset of some $F_i$. Without loss of generality, assume $\tilde F_k$ is nonempty for each $k$. Let
$\beta^F=(\beta_1^F,\ldots,\beta_{m+1}^F)^\ast$, $\beta^{\tilde F}=(\beta_1^{\tilde F},\ldots,\beta_{n+1}^{\tilde F})^\ast$, $(q_{ij}^F)$ and $(q_{kl}^{\tilde F})$ be defined similarly by (\ref{beta}) and (\ref{Q-f}).

\begin{prop}\label{refine}
Suppose that for each $k$, $1\leq k\leq n+1$,
\begin{equation}\label{q-ff}
\begin{split}
q_{ij}^F&\geq \sum_{l:\tilde F_l\subseteq F_j} q_{kl}^{\tilde F}, \ \text{if}\ i<j;\quad
q_{ij}^F\leq \sum_{l:\tilde F_l\subseteq F_j} q_{kl}^{\tilde F},\ \text{if}\ i>j.
\end{split}
\end{equation}
Then the fact $-(Q^F+\diag(\beta^F))H_{m+1}$ is a nonsingular M-matrix yields that so is the matrix $-(Q^{\tilde F}+\diag(\beta^{\tilde F}))H_{n+1}$.
\end{prop}

\begin{proof}
According to Proposition \ref{m-matrix}, as $-(Q^F+\diag(\beta^F))H_{m+1}$ is a nonsingular M-matrix, there exists a vector $\eta^F\gg 0$ such that
$\big(Q^F+\diag(\beta^F))H_{m+1}\eta^F\ll 0$.
Let $\xi^F=H_{m+1}\eta^F$, then $\xi^F\gg 0$ and $\xi_i^F\leq \xi_{i+1}^F$ for $i=1,\ldots,m$.
Let $\xi_k^{\tilde F}=\xi_i^F$ if $\tilde F_k\subseteq F_i$, $k=1,\ldots, n+1$, and $\xi^{\tilde F}=\{\xi_1^{\tilde F},\ldots, \xi_{n+1}^{\tilde F}\}$.
Then by (\ref{q-ff}) and the fact $\beta_k^{\tilde F}\leq \beta_i^F$ if $\tilde F_k\subseteq F_i$, we have,
\begin{align*}
&\big(Q^{\tilde F}+\diag(\beta^{\tilde F})\xi^{\tilde F}
=\sum_{l=1}^{n+1}q_{kl}^{\tilde F}\xi_l^{\tilde F}+\beta_k^{\tilde F}\xi_k^{\tilde F}\\
&=\sum_{j<i}\big(\sum_{l:\tilde F_l\subseteq F_j}q_{kl}^{\tilde F}\big)(\xi_j^F-\xi_i^F)+\sum_{j>i}\big(\sum_{l:\tilde F_l\subseteq F_j}q_{kl}^{\tilde F}\big)
(\xi_j^F-\xi_i^F)+\beta_k^{\tilde F}\xi_i^F\\
&\leq \sum_{j\neq i}q_{ij}^F(\xi_j^F-\xi_i^F)+\beta_i^F\xi_i^F\\
&=\big(Q^F+\diag(\beta^F)\big)\xi^F(i)\ll 0.
\end{align*}
Therefore, $-\big(Q^{\tilde F}+\diag(\beta^{\tilde F})H_{n+1}$ is also a nonsingular M-matrix due to Proposition \ref{m-matrix}.
\end{proof}

By Theorem \ref{t-m-infi}, we can provide some examples of regime-switching  processes in an infinite state space, which are exponentially ergodic in the Wasserstein distance $W_{\tilde \rho}$.

\begin{exam} Let $\S=\{1,2,\ldots\}$ be a countable set. Let $(X_t,\La_t)$ be a state-independent regime-switching diffusion process given by (\ref{1.1}) and (\ref{1.2}). Assume (A1-A4) hold. Let $F_1=\{j\in\S; \beta_j<0\}$ and $F_2=\{j\in \S;\beta_j>0\}$. Set $\beta_1^F=\sup_{j\in F_1} \beta_j$ and $\beta_2^F=\sup_{j\in F_2}\beta_j$. $Q^F=(q_{ij}^F)$ is induced from $Q$ as above. We now check the condition that
\[-\big(Q^F+\diag(\beta_1^F,\beta_2^F)\big)H_2=\begin{pmatrix}-q_{11}^F-\beta_1^F& -\beta_1^F\\ q_{22}^F& -\beta_2^F\end{pmatrix}\]
is a nonsingular M-matrix.
By Proposition \ref{m-matrix}, it is equivalent to
\[\beta_1^F<-q_{11}^F\ \text{and}\ \beta_1^F<\beta_2^F<\frac{q_{22}^F\beta_1^F}{-q_{11}^F-\beta_1^F}.\]
This ensures that there are many regime-switching diffusion processes $(X_t,\La_t)$ with infinite state space $\S$ such that the conditions of Theorem \ref{t-m-infi} hold.
\end{exam}
 Next, we provide a more concrete example.
\begin{exam}
Let $(\La_t)$ be a birth-death process on countable set $\S=\{1,2,\ldots\}$. For each $i>1$,  set $q_{i i+1}=b_i>0$ and $q_{i i-1}=a_i>0$, and $q_{ij}=0$ for $j\neq i+1$ or $i-1$. Let $q_{12}=b_1>0$. Set $\mu_1=1$ and $\mu_n=b_1b_2\cdots b_{n-1}/a_2a_3\cdots a_n$ for $n\geq 2$. Assume
\[\sum_{i=1}^\infty \frac{1}{\mu_ib_i}\sum_{j=i+1}^\infty \mu_j<\infty.\]
Let $(\La_t,\La_t')$ be the classical coupling whose generator is given by
\begin{equation*}
\tilde Qh(i,j)=\begin{cases} [a_i(h(i-1,j)-h(i,j))+b_i(h(i+1,j)-h(i,j))]\\
                              \quad +[a_j(h(i,j-1)-f(i,j))+b_j(h(i,j+1)-h(i,j))], \quad i\neq j,\\
                              a_i(h(i-1,j-1)-h(i,j))+b_i(h(i+1,j+1)-h(i,j)), \quad i=j.
                \end{cases}
\end{equation*}
Let
\[g(i,j)=\sum_{k=1}^{j-1}\frac{1}{\mu_k b_k}\sum_{l=k+1}\mu_l.\]
Then $g$ satisfies
$\displaystyle \|g\|_\infty:=\sup_{(i,j)\in\S^2} g(i,j)\leq \sum_{k=1}^\infty\frac{1}{\mu_k b_k}\sum_{l=k+1}^\infty \mu_l<\infty$ by assumption.
It is easy to check that  $\tilde Q g(i,j)=-1$ by direct calculation. Therefore, assumption (A4) is satisfied.

For each $i\geq 1$, let $(X_t^{(i)})$ be a diffusion process on $[0,\infty)$ with reflecting boundary at $0$ satisfying following SDE:
\[\d X_t^{(i)}=\beta_i X_t^{(i)}\d t+\sqrt{2}\d B_t,
\]
where $\beta_i$ is a constant and $(B_t)$ is a Brownian motion. When $\beta_i<0$, $(X_t^{(i)})$ is an Ornstein-Uhlenbeck process, which is exponential ergodic. But when $\beta_i>0$, $(X_t^{(i)})$ is not recurrent.
For each $i\geq 1$, define a reflecting coupling for $(X_t^{(i)})$ with infinitesimal generator $\tilde L^{(i)}\sim (a^{(i)}(x,y),b^{(i)}(x,y))$, where
\[a^{(i)}(x,y)=\begin{pmatrix}1& -1\\ -1& 1\end{pmatrix},\quad b^{(i)}(x,y)=\begin{pmatrix}\beta_i\, x\\ \beta_i\, y\end{pmatrix}.\]
Let $\rho(|x-y|)=|x-y|$, then it is easy to see
\[\tilde L^{(i)}\rho(|x-y|)=\beta_i\rho(|x-y|),\quad x\neq y.\]
Therefore, Assumption (A1) holds. Let $\beta_1=-\kappa_1$ and $\beta_i=\kappa_2-i^{-1}$ for $i\geq 2$, where $\kappa_1$ and $\kappa_2$ are two positive constants.

Let $(X_t)$ be a solution of the following SDE:
\[\d X_t=\beta_{\La_t}X_t\d t+\sqrt{2}\d B_t, X_0=x>0.
\]
Then $(X_t,\La_t)$ is a state-independent regime-switching diffusion process satisfying assumptions (A1-A4).
Take $F_1=\{1\}$ and $F_2=\{2,3,\ldots\}$, which is a finite partition of $\S=\{1,2,\ldots\}$.
Then  $\beta_1^F=\beta_1=-\kappa_1$,
$\beta_2^F =\kappa_2$, $q^F_{12}=\sum_{j\in F_2}q_{1j}=b_1$ and $q^F_{21}=\sup_{i\in F_2} q_{i1}=a_2$.
When
$\dis
\kappa_2<\frac{a_2\kappa_1}{b_1+\kappa_1}$,
the matrix $-(Q^F+\diag(\beta_1^F,\beta_2^F)H_2$ is a nonsingular M-matrix.
Hence, the regime-switching process $(X_t,\La_t)$ is \textbf{exponentially ergodic} in the Wasserstein distance $W_{\tilde \rho}$ with
$\tilde \rho((x,i),(y,j))=\sqrt{\mathbf 1_{i\neq j}+|x-y|}$, if  \[\kappa_2<\frac{a_2\kappa_1}{b_1+\kappa_1}.\]
This example shows that although the diffusion process $(X_t)$ in a random environment characterized by $(\La_t)$ is transient in infinitely many  environments ($i\geq 2$), and is recurrent only in a environment ($i=1$), the process $(X_t)$ could be recurrent.
\end{exam}

\section{State-dependent switching in an infinite state space}
In this section, we study state-dependent regime-switching diffusion processes $(X_t,\La_t)$ defined by (\ref{1.1}) and (\ref{1.2}), that is, the Q-matrix of $(\La_t)$ depends on $(X_t)$. This makes the coupling process used in previous two sections useless because we can not make the coupling process $(\La_t,\La_t')$ moves together after their first meeting. So it is difficult in this case to construct successful coupling $(Y_t,\La_t')$ of $(X_t,\La_t)$ to estimate the Wasserstein distance between them. In \cite{XS}, we discussed how to construct successful couplings for state-dependent regime-switching process with $\S$ being finite. In this section, we shall study the asymptotic boundedness of $(X_t,\La_t)$. We extend the known results to state-dependent regime-switching diffusion processes in an infinite state space.

In this section,  $\S$ is an infinite set, i.e. $N=\infty$. Let $\rho:[0,\infty)\ra [0,\infty)$ satisfying $\rho(0)=0$, $\rho'>0$, $\lim_{r\ra \infty} \rho(r)=\infty$. We assume that
\begin{itemize}
\item[(H)] For each $i\in\S$, there exists  constants $\theta_i\in \R$, $K_i\in[0,\infty)$ such that
 \[L^{(i)}\rho(|x|)\leq \theta_i\rho(|x|)+K_i,\quad x\in\R^d,\] and $M_1:=\sup_{i\geq 1} \theta_i<\infty$, $M_2:=\sup_{i\geq 1} K_i<\infty$.
\end{itemize}
Divide $\S$ into finite nonempty subsets according to $\theta_i$. Let
\[\Gamma:=\{-\infty=:k_0<k_1<\cdots<k_m<k_{m+1}=M_1\}.\]
Corresponding to $\Gamma$, there is finite partition of $\S$, denoted by
$F=\{F_1,\ldots,F_{m+1}\}$, where
\[F_i=\{j\in \S;\ \theta_j\in (k_{i-1},k_i]\}.
\]
Let $\phi:\S\ra \{1,\ldots,m+1\}$ be defined by $\phi(j)=i$ if $j\in F_i$.
Set
\[\theta_i^F=\sup_{j\in F_i}\theta_j,\quad \text{for $i=1,\ldots,m+1$}, \ \text{so}\ \theta_j\leq \theta_{\phi(j)}^F\ \ \text{and $\theta_i^F<\theta_{i+1}^F$}.\]
Define a new $Q$-matrix $Q^F=(q_{ij}^F)$ on the space $\{1,\ldots,m+1\}$ corresponding to partition $F$ by
\begin{equation}\label{b-q-f}
q_{ik}^F=\inf_{r\in F_i}\inf_{x\in\R^d} \sum_{j\in F_k} q_{rj}(x), \ k>i;\quad q_{ik}^F=\sup_{r\in F_i}\sup_{x\in\R^d}\sum_{j\in F_k} q_{rj}(x),\ k<i,\ \ q_{ii}^F=-\sum_{k\neq i} q_{ik}^F.
\end{equation}

\begin{thm}\label{b-t-4}
Assume that (H) holds. Use the notation defined above. If the $(m+1)\times (m+1)$ matrix $-\big(Q^F+\diag(\theta_1^F,\ldots,\theta_{m+1}^F)\big)H_{m+1}$ is a nonsingular M-matrix, where
$H_{m+1}$ is defined by (\ref{h-matrix}),
then there are constants $\alpha,\,c_1,\,c_2>0$ such that
\begin{equation*}
  \E[\rho(|X_t|)]\leq c_1\E[\rho(|X_0|)]e^{-\alpha t}+c_2,\quad t>0.
\end{equation*}
\end{thm}
\begin{proof}
As $-\big(Q^F+\diag(\theta_1^F,\ldots,\theta_{m+1}^F)\big)H_{m+1}$ is a nonsingular M-matrix, there is a vector $\eta^F=(\eta_1^F,\ldots,\eta_{m+1}^F)^\ast\gg 0$ such that
\[\lambda^F=(\lambda_1^F,\ldots,\lambda_{m+1}^F)^\ast:
=(Q^F+\diag(\theta_1^F,\ldots,\theta_{m+1}^F))H_{m+1}\eta^F\ll 0.\]
Then $\lambda_{\mathrm{max}}:=\max_{1\leq i\leq m+1} \lambda_i^F<0$. Set $\xi^F=H_{m+1}\eta^F$. It is easy to see
\[\xi_i^F=\eta_{m+1}^F+\cdots+\eta_i^F,\ i=1,\ldots,m+1.\]
Hence, $\xi_{i+1}^F<\xi_i^F$, $i=1,\ldots, m$, and $\xi^F\gg 0$.
To proceed, we extend $\xi^F$ to a vector on $\S$ by setting $\xi_j=\xi_i^F$ if $j\in F_i$. Then we have, for $r\in F_i$, $x\in\R^d$,
\begin{align*}
Q_x\xi(r)&=\sum_{j\in\S,j\neq r} q_{rj}(x)(\xi_j-\xi_r)=\sum_{j\in F_i,j\in \S} q_{rj}(x)(\xi_j-\xi_r)\\
&=\sum_{k< i}\big(\sum_{j\in F_k} q_{rj}(x)\big)(\xi_k^F-\xi_i^F)+\sum_{k> i}\big(\sum_{j\in F_k} q_{rj}(x)\big)(\xi_k^F-\xi_i^F)\\
&\leq \sum_{k\neq i} q_{ik}^F(\xi_k^F-\xi_i^F)=\big(Q^F\xi^F\big)(i).
\end{align*}
By It\^o's formula, we obtain
\begin{align*}
&\E[\rho(|X_t|)\xi_{\La_t}]\\
&\leq \E[\rho(|X_0|)\xi_{\La_0}]
+\E\Big[\int_0^t\big((Q_{X_s}\xi)(\La_s)+\theta_{\La_s}\xi_{\La_s}\big)\rho(|X_s|)
+K_{\La_s}\xi_{\La_s}\d s\Big]\\
&\leq \E[\rho(|X_0|)\xi_{\La_0}]
+\E\Big[\int_0^t\big((Q^F\xi^F)(\phi(\La_s))+\theta_{\phi(\La_s)}^F
\xi_{\phi(\La_s)}^F\big)\rho(|X_s|)+K_{\La_s}\xi_{\phi(\La_s)}^F\d s\Big]\\
&\leq \E[\rho(|X_0|)\xi_{\La_0}]+\lambda_{\mathrm{max}}^F\E\Big[\int_0^t\rho(|X_s|)\d s\Big]
+M_2\xi_{\mathrm{max}}^Ft.
\end{align*}
This yields that
\begin{equation*}
  \E[\rho(|X_t|)\xi_{\La_t}]\leq \E[\rho(|X_0|)\xi_{\La_0}]e^{\frac{\lambda_{\mathrm{max}}^F}{\xi_{\mathrm{max}}^F} t}-\frac{M_2(\xi_{\mathrm{max}}^F)^2}{\lambda_{\mathrm{max}}}.
\end{equation*}
Therefore,
\begin{equation*}
  \E[\rho(|X_t|)]\leq \E[\rho(|X_0|)]\frac{\xi_{\mathrm{max}}^F}{\xi_{\mathrm{min}}^F}
  e^{\frac{\lambda_{\mathrm{max}}^F}{\xi_{\mathrm{mx}}^F} t}-\frac{M_2(\xi_{\mathrm{max}}^F)^2}{\lambda_{\mathrm{max}}\xi_{\mathrm{min}}^F},\ \ t>0.
\end{equation*}
We conclude the proof by taking $c_1=\frac{\xi_{\mathrm{max}}^F}{\xi_{\mathrm{min}}^F}$, $\alpha=-\frac{\lambda_{\mathrm{max}}^F}{\xi_{\mathrm{mx}}^F}$ and $c_2=-\frac{M_2(\xi_{\mathrm{max}}^F)^2}{\lambda_{\mathrm{max}}\xi_{\mathrm{min}}^F}$.
\end{proof}

\end{document}